\documentclass[12pt]{article}%
\usepackage{amsmath}
\usepackage{amsfonts}
\usepackage{amssymb}
\usepackage{graphicx}%
\providecommand{\U}[1]{\protect\rule{.1in}{.1in}}
\newtheorem{theorem}{Theorem}

\newtheorem{example}[theorem]{Example}

\newtheorem{lemma}[theorem]{Lemma}

\newtheorem{proposition}[theorem]{Proposition}

\newenvironment{proof}[1][Proof]{\noindent\textbf{#1.} }{\ \rule{0.5em}{0.5em}}
\begin{document}

\title{A Nakano carrier type-theorem for orthogonally additive polynomials in Riesz spaces}
\author{Elmiloud Chil and Khansa Weslati\\Elmiloud Chil\\Institut pr\'{e}paratoire aux \'{e}tudes d'ing\'{e}nieurs de Tunis \\2 Rue Jawaher lel Nahrou Montfleury 1008 TUNISIA\\University of Tunis\\Elmiloud.chil@ipeit.rnu.tn\\Khansa Oueslati\\L.A.T.A.O Faculty of Sciences of Tunis Elmanar\\Tunis Elmanar University \\Khansaoueslati@ymail.com}
\date{}
\maketitle

\begin{abstract}
The purpose of the present paper is to prove the Nakano theorem for
orthogonally additive polynomials in Riesz spaces

\end{abstract}



\section{Introduction and Preliminaries}

We take it for granted that the reader is familiar with notions of Riesz
spaces (or vector lattices) and operators between them. For terminology,
notations and concepts that are not explained in this paper, one can refer to
the standard monographs [2,15,19].

A Riesz space $E$ is called \textit{Archimedean} if for each non zero $a\in E$
the set $\left\{  na,n=\pm1,\pm2,...\right\}  $ has no upper bound in $E$. In
order to avoid unnecessary repetition we will assume throughout the paper that
all Riesz spaces under consideration are Archimedean. Recall that a subset $I$
of a Riesz space is called solid whenever $\left\vert x\right\vert
\leq\left\vert y\right\vert $ and $y\in I$ imply $x\in I$. A solid vector
subspace is called an ideal and an order closed ideal is referred to as a
band. Thus an ideal $I$ is a band \ if and only if $x_{\alpha}\in I$ and
$0\leq x_{\alpha}\nearrow x$ imply $x\in I$. A vector sublattice $I$ of $E$ is
said to be order dense in $E$ whenever for every $0<x\in E$ there exists $y\in
I$ satisfying $0<y\leq x$ or, equivalently, there exists an upward directed
net $0\leq y_{\alpha}\in I$ such that $y_{\alpha}\nearrow x$. A Riesz space is
called \textit{Dedekind complete} whenever every nonempty subset that is
bounded from above has a supremum. Similarly, a Riesz space is said to be
$\sigma$\textit{-Dedekind complete} whenever every countable subset that is
bounded from above has a supremum. Let us say that a vector subspace $G$ of a
Riesz space $E$ is \textit{majorizing} $E$ whenever for each $x\in E$ there
exists some $y\in G$ with $x\leq y$.(or, equivalently whenever for each $x\in
E$ there exists some $y\in G$ with $y\leq x$)$.$ A Dedekind complete Riesz
space $L$ is said to be a \textit{Dedekind completion} of the Riesz space $E$
whenever $E$ is Riesz isomorphic to an order dense majorizing Riesz subspace
of $L$ (which we identify with $E$). It is a classical result that every
Archimedean Riesz space $E$ has a Dedekind completion, which we shall denote
by $E^{\delta}$. The closure, with respect to the relatively uniform topology
of the Riesz space $E$ in its Dedekind completion is a uniformly complete
vector lattice, denoted by $E^{ru}$ and referred to as the uniform completion
of $E$ (see [15]). Thus $E$ is an order dense majorizing Riesz subspace of
$E^{ru}$. The Riesz space $E$ is said to be a lattice ordered algebra
(briefly, an $l$-algebra) if there exists an associative multiplication in $E$
with the usual algebra properties such that $xy\in E^{+}$ for all $x,y\in
E^{+}$. The $l$-algebra $E$ is called $f$-algebra if $E$ has the property that
$x\wedge y=0$ in $E$ implies $zx\wedge y=xz\wedge y=0$ for all $z\in E^{+}$ (
for more details about $f$-algebras see [16]). A Dedekind complete Riesz space
is called universally complete whenever every set of pairwise disjoint
positive elements has a supremum. Every Riesz space $E$ has a universally
completion $E^{u}$, i.e., there exists a unique (up to Riesz homomorphism)
universally complete ( and therefore Dedekind complete) Riesz space $E^{u}$ so
that $E$ can be identified with an order dense Riesz subspace of $E^{u}$.
Moreover, $E^{u}$ is furnished with a multiplication, under which $E^{u}$ is
an $f$-algebra with unit element (see [2]).

Let $E$ be a Riesz space and let $F$ be a vector space. A map $P:E\rightarrow
F$ is called a \textit{homogeneous polynomial} of degree $n$ (or a
$n$-homogeneous polynomial) if $P(x)=\psi(x,..,x)$, where $\psi$ is a
$n$-multilinear map from $E^{n}$ into $F$. Each homogeneous polynomial has an
unique associated symmetric multilinear mapping. A homogeneous polynomial, of
degree $n,$ $P:E\rightarrow F$ is said to be \textit{orthogonally additive} if
$P(x+y)=P(x)+P(y)$ where $x,y\in E$ are orthogonally (i.e. $\left\vert
x\right\vert \wedge\left\vert y\right\vert =0$) or equivalent that its
symmetric multilinear mapping $\psi$ being orthosymmetric, that is,
$\psi(x_{1},..,x_{n})=0$ whenever $x_{1},..,x_{n}\in E$ satisfy $x_{i}\perp
x_{j}$ for some $i\neq j$. We denote by $\mathcal{P}_{0}(^{n}E,F)$ the set of
$n$-homogeneous orthogonally additive polynomials from $E$ to $F$.

The study of orthogonally additive polynomials is of interest both from the
algebraic point of view and also from the point of view of infinite
dimensional analysis, in particular the theory of holomorphic functions on
infinite dimensional analysis. To the best of our knowledge the first
mathematician interested in orthogonally additive polynomials was Sundaresan
[20] who obtained a representation theorem for polynomials on $\ell^{p}$ and
on $L^{p}$. It is only recently that the class of such mappings has been
getting more attention. We are thinking here about works on orthogonally
additive polynomials and holomorphic functions and orthosymmetric multilinear
mappings on different Banach lattices (see [4,6,7,9,11,14,18,20]), on $%
\mathbb{C}
^{\ast}$-algebras (see [1,17]), on uniformly complete vector lattices
(see[3,8,21]), on uniformly complete vector lattices with range space is
separated convex bornological spaces (see [12]), also on uniformly complete
vector lattices taking values in Hausdorff topological vector spaces (see [3])
and on Riesz spaces taking values in Hausdorff topological vector spaces (see
[10]). Notice that the motivation for the study of this class of maps came
from one of the relevant problems in operator theory which is to describe
homogeneous orthogonally additive polynomial via linear operators. This is
called the representation theorem in Riesz spaces and which, with their
applications, will give a particular interest to the space $\mathcal{P}%
_{o}(^{n}E,F)$ in functional analysis. It finds its motivations in many
mathematical problems, coming from industrial applications, economic modeling,
abstract mathematical questions. In several situations, it is convenient to
substitute the study of this asymptotic with the study of the so-called
"representation theorem problem". Naturally, this representation theorem
problem must capture the main behavior of the Riesz topologies of
$\mathcal{P}_{o}(^{n}E,F)$ and a family of problems being dealt with, and its
solutions need to be more easily obtained. This problem can be treated in a
different manner, depending on domains and co-domains on which polynomials
act. Next, we give a short historical account for this problem. In 1990,
Sundaresan in [20] proved that every $n$-homogeneous orthogonally additive
polynomials $P:L^{p}\rightarrow%
\mathbb{R}
$ is determined by some $g\in L^{\frac{p}{p-n}}$ via the formula $P(f)=%
{\displaystyle\int}
f^{n}gd_{\mu}$ for all $f\in L^{p}$ and $n\leq p$. Thus $\mathcal{P}_{o}%
(^{n}L^{p},%
\mathbb{R}
)$ is lattice isomorphie to $L^{\frac{p}{p-n}}$. In the case of the Reisz
space $l_{1}$, $P$ is an order bounded orthogonally additive $n$-homogeneous
polynomial if and only if there exists a bounded sequence of real numbers
$(a_{j})$, such that
\[
P((x_{n}))=\overset{+\infty}{\underset{i=0}{\sum}}a_{j}x_{j}^{n}%
\]
for all $(x_{n})\in l_{1}$, which gives a canonical lattice isomorphism
between $\mathcal{P}_{ob}(^{n}l_{1},%
\mathbb{R}
)$ and $l_{\infty}$, the space of all bounded sequence of real numbers. The
next significant development was the discovery of an integral representation
for homogeneous orthogonally additive polynomials on $C(X)$ spaces by D.
Prez-Garcia and I.Villanueva in [18], D.Carando, S.Lassalle and I. Zalduendo
in [9], who proved a representation of the form $P(f)=%
{\displaystyle\int}
f^{n}d_{\mu}$ for all $f\in C(X)$ where $\mu$ is a regular Borel signed
measure on $X$. In [4] a result analogous to that of K. Sundareson for the
classes of Banach lattices of functions of order continuous K\"{o}the function
space, whose dual is given by integrals, have proven by Y. Benyamini, S.
Lassalle and LIavona. Proofs of the aforementioned results are strongly based
on the representation of Riesz spaces as vector spaces of extended continuous
functions. So they are not applicable to general Riesz spaces. Therefore, it
seems to be an interesting question what about representation theorem in
general case of Riesz spaces. The study of homogeneous orthogonally additive
polynomials has been greatly developed these last years see [3,6,10,21]. In
[10] authors showed that this kind of behavior is avoided by studying the
multilinear operators on Riesz spaces. In spite of the results mentioned
above, they are looking for a natural and explicit representation of a
homogeneous orthogonally additive polynomial on Riesz space. Such space would
play the role of the $n$-concavification of Banach lattices used in the
representation theorem for homogeneous orthogonally additive polynomials on
such spaces. More precisely they prove that if $E$ is an Archimedean Riesz
space, $F$ is a Hausdorff topological vector space and $P\in\mathcal{P}%
_{o}(^{n}E,F)$ then there exists a unique linear operator $T_{p}%
:\underset{i=1}{\overset{n}{\Pi}}E^{ru}\rightarrow F$ (which will be called
the associated operator of $P$) such that
\[
P(x)=T_{P}(x^{n})
\]
for all $x\in E$, where $E^{ru}$ is the uniform completion of $E$and the
product $\underset{i=!}{\overset{n}{\Pi}}E^{ru}=\left\{  \underset
{i=1}{\overset{n}{\pi}}x_{i}:x_{i}\in E\right\}  $ is the Riesz subspace of
$E^{u}$, the universal completion of $E$. Here, the multiplication is the
$f$-algebra multiplication of $E^{u}$, see [5,10] for more details. Perhaps
the most striking work in this direction, due to the first author and Dorai in
[10] in which they give by using the last mentioned representation theorem for
$P\in\mathcal{P}_{o}(^{n}E,F)$ a complete description for the lattice
structure of $\mathcal{P}_{ob}(^{n}E,F)$, the space of order bounded
$n$-homogeneous orthogonally additive polynomials in the case when \ $F$ is a
Dedekind complete Riesz space. That is, if $P\in\mathcal{P}_{ob}(^{n}E,F)$
then $T_{P}$ is an order bounded operator. This result gave rise of the
concept of lattice structure of $\mathcal{P}_{ob}(^{n}E,F)$ via the formula:%
\[
\left\vert P\right\vert (x)=\left\vert T_{P}\right\vert (x^{n})\text{ for all
}x\in E^{+}.
\]
Which gives a canonical lattice isomorphism between $\mathcal{P}_{ob}(^{n}E,%
\mathbb{R}
)$ and $\mathcal{L}_{b}(\underset{i=!}{\overset{n}{\Pi}}E^{ru},F)$. This
representation theorem without exception will be of great importance to us
here. As an application and by using the order continuity of the $f$-algebra
product it is easy to see that $P\in\mathcal{P}_{ob}(^{n}E,F)$ is order
continuous if and only if its associated operator $T_{P}$ is order continuous.
Which implies that for orthogonally additive polynomials order continuity at
the origin implies order continuity at every point. Let us note that this
property is not true in general case of homogeneous polynomials ( for more
details see [6]). We denote by $\mathcal{P}_{oc}(^{n}E,F)$ the set of order
continuous $n$-homogeneous orthogonally additive polynomials from $E$ to $F$.
Obviously $\mathcal{P}_{oc}(^{n}E,F)\subset\mathcal{P}_{ob}(^{n}E,F).$

In this paper, we focus on the carrier theorem for order continuous
homogeneous orthogonally additive polynomials in Riesz spaces that generalizes
the classical Nakano Carrier theorem for linear functionals (see [2]). Some
generalization to this concept and its applications was obtained in by C.
Boyed, A. Raymond \ and N. Snigireva in [6] in which they give the statement
of the Nakano Carrier theorem for order continuous homogeneous orthogonally
additive polynomials in Banach lattices. Proofs of the aforementioned results
are strongly based on the representation of these spaces as vector spaces of
extended continuous functions by using localization techniques. So they are
not applicable to general Riesz spaces. The present paper is largely motivated
by this work. In this sense, we intend to make some contributions to this area
by proving a Riesz space version of the Nakano Carrier Theorem for order
continuous homogeneous orthogonally additive polynomials. Actually, Banach
lattices will be replaced by Riesz spaces. Note also that the current paper
emphasizes the algebraic aspects of the problem. That is why, we need to
develop new approaches. Actually, we provide not only new results but also new
techniques, which we think are useful additions to the literature.
Nonetheless, we are not aware of works fully dedicated to this kind of
problems in the framework of Riesz space. Thus, to our knowledge, the study
which we carry out in this paper is original.

\section{Main Results}

In order to avoid unnecessary repetition we will assume throughout the paper
that $E$ is an Archimedean Riesz space and $F$ is a Dedekind complete Riesz spaces.

Before we pass to the details of this part we first refresh the memory by
recalling the definitions of the null ideal and the carrier of an homogeneous
orthogonally additive polynomial. The null ideal of an arbitrary
$P\in\mathcal{P}_{ob}(^{n}E,F)$ is the set%
\[
N_{P}=\left\{  x\in E:\left\vert P\right\vert (\left\vert x\right\vert
)=0\right\}  =\left\{  x\in E:\left\vert T_{P}\right\vert (\left\vert
x^{n}\right\vert )=0\right\}
\]
and its carrier is the band
\[
C_{P}=N_{P}^{d}\text{.}%
\]
Thus the carrier is a natural well behaved object essentially with respect to
order continuity for $P\in\mathcal{P}_{ob}(^{n}E,F).$ Let us note that
$N(P)=N(\left\vert P\right\vert $ so we emphasize that we can assume that $P$
is positive.

We start this section by the following Lemma in which we prove that every
positive homogeneous orthogonally additive polynomial is increasing on $E^{+}$.

\begin{lemma}
Let $P\in(\mathcal{P}_{ob}(^{n}E,F))^{+}$ and $0\leq x\leq y$ then $0\leq
P(x)\leq P(y).$
\end{lemma}

\begin{proof}
Let $P\in(\mathcal{P}_{ob}(^{n}E,F))^{+}$ and $T_{P}$ its associated operator.
If $0\leq x\leq y$ then $0\leq x^{n}\leq y^{n}$ from the positivity of $T_{P}$
we have $0\leq T_{P}(x^{n})=P(x)\leq T_{P}(y^{n})=P(y).$
\end{proof}

The main results of this paper is strongly based on the following Theorem.

\begin{theorem}
$N_{P}$ is an ideal of $E$ and its a band whenever $P$ is order continuous.
\end{theorem}

\begin{proof}
As mentioned in the beginning of this part we can assume that $P$ and then
$T_{P}$ are positive. It is only we need to prove that $N_{p}$ is a vector
space. We begin by proving the case where $n$ is even. Let $x,y\in(N_{P})^{+}$
then we have that
\[
P(\lambda x+y)=T_{P}(\lambda x+y)^{n}=\overset{n-1}{\underset{k=1}{%
{\displaystyle\sum}
}}C_{n}^{k}\lambda^{k}T_{P}(x^{k}y^{n-k})\geq0
\]
for all $\lambda\in\mathcal{%
\mathbb{R}
}$ which implies that
\[
T_{P}(x^{k}y^{n-k})=0\text{ for all }1\leq k\leq n-1\text{.}%
\]
Thus $P(\lambda x+y)=0$ for all $\lambda\in\mathcal{%
\mathbb{R}
}$ and $x,y\in(N_{P})^{+}$. Now from the preceding Lemma we have
\[
0\leq P(\left\vert \lambda x+y\right\vert )\leq P(\left\vert \lambda
\right\vert \left\vert x\right\vert +\left\vert y\right\vert )=0\text{ for all
}x,y\in N_{P}\text{ and }\lambda\in\mathcal{%
\mathbb{R}
}%
\]
So $P(\left\vert \lambda x+y\right\vert )=0$. Therefore $\lambda x+y\in N_{P}$
for all $x,y\in N_{P}$ and $\lambda\in\mathcal{%
\mathbb{R}
}.$ Secondly if $n$ is odd. Let $x,y\in(N_{P})^{+}$ then we have
\[
P(\lambda x+y)^{2}=(T_{P}(\lambda x+y)^{n})^{2}=\overset{n-1}{\underset{k=1}{%
{\displaystyle\sum}
}}\overset{n-1}{\underset{j=1}{%
{\displaystyle\sum}
}}C_{n}^{k}C_{n}^{j}\lambda^{k+j}T_{P}(x^{j}y^{n-j})T_{P}(x^{k}y^{n-k}%
)\geq0\text{ }%
\]
for all $\lambda\in\mathcal{%
\mathbb{R}
}$ (where here the product is of the $f$-algebra $F^{u}$) which implies that
\[
T_{P}(x^{j}y^{n-j})T_{P}(x^{k}y^{n-k})=0\text{ for all }1\leq k\leq n-1\text{
and }1\leq j\leq n-1\text{.}%
\]
In particular for $j=k,$ we have
\[
T_{P}(x^{k}y^{n-k})=0\text{ for all }1\leq k\leq n-1\text{.}%
\]
because $F^{u}$ is semiprime. Thus $P(\lambda x+y)=0$ for all $\lambda
\in\mathcal{%
\mathbb{R}
}$ and $x,y\in N_{P}$. Since by the preceding lemma
\[
0\leq P(\left\vert \lambda x+y\right\vert )\leq P(\left\vert \lambda
x+y\right\vert )\leq P(\left\vert \lambda\right\vert \left\vert x\right\vert
+\left\vert y\right\vert )=0
\]
for all $x,y\in N_{P}$ and $\lambda\in\mathcal{%
\mathbb{R}
}$. We deduce that $P(\left\vert \lambda x+y\right\vert )=0$. Therefore
$\lambda x+y\in N_{P}$ for all $x,y\in N_{P}$ and $\lambda\in\mathcal{%
\mathbb{R}
}$ and the proof is finished.
\end{proof}

All the preparations have been made for the first central result in the paper.
H. Nakano in [2] has shown that two linear functionals are disjoint if and
only if their carriers are disjoint sets. In the next result we prove a Nakano
type theorem for order continuous $n$-homogeneous orthogonally additive polynomials.

\begin{theorem}
Let $P,Q\in\mathcal{P}_{ob}(^{n}E,%
\mathbb{R}
)^{+}$ such that $P$ or $Q$ is order continuous. Then the following are equivalents

\begin{enumerate}
\item $P\perp Q$

\item $C_{P}\subset N_{Q}$

\item $C_{Q}\subset N_{P}$

\item $C_{P}\perp C_{Q}$
\end{enumerate}
\end{theorem}

\begin{proof}
First note that by symmetry we can assume that $Q$ is order continuous. For
the implication $1\Rightarrow2$ let $0\leq x\in C_{P}=N_{P}^{d}$ from the fact
that $P\wedge Q=0$ we have $P\wedge Q(x)=T_{P}\wedge T_{Q}(x^{n})=0$ Now by
the Riesz Kantorovic theorem (see [2])we have
\[
T_{P}\wedge T_{Q}(x^{n})=\inf\left\{  T_{P}(y^{n})+T_{Q}(z^{n}):y,z\in
E^{+}\text{ such that }y^{n}+z^{n}=x^{n}\right\}  =0.
\]
So let $\epsilon>0$ there exists a sequences $(x_{k})\in E^{+}$ satisfying
\[
0\leq x_{k}^{n}\leq x^{n}\text{, so }0\leq x_{k}\leq x,\text{ and }T_{P}%
(x_{k}^{n})+T_{Q}(x^{n}-x_{k}^{n})<\frac{\epsilon}{2^{k}}\text{.}%
\]
Put $y_{k}=\underset{i=1}{\overset{k}{\wedge}}x_{i}$ and note that
$y_{k}\searrow0$ in $E$ indeed $0\leq y\leq y_{k}$ then $0\leq y^{n}\leq
y_{k}^{n}$ in $\underset{i=!}{\overset{n}{\Pi}}E^{ru}$ \ So by Lemma 1 we
have
\[
0\leq P(y)=T_{P}(y^{n})\leq P(y_{k})=T_{p}(y_{k}^{n})\leq P(x_{k}%
)<\frac{\epsilon}{2^{k}}\text{ for all }k\in%
\mathbb{N}
\text{.}%
\]
Consequently $P(y)=0$. Thus $y\in C_{P}\cap N_{P}=\left\{  0\right\}  $.so
$y=0.$ Now since $T_{Q}$ is order continuous (because $Q$ is order continuous)
and $y_{k}^{n}\searrow0$ in $\underset{i=!}{\overset{n}{\Pi}}E^{ru}$ (because
the product of $f$-algebra is order continuous) we have $T_{Q}(x^{n}-y_{k}%
^{n})\nearrow T_{Q}(x^{n})=Q(x)$. On the other hand
\begin{align*}
0  &  \leq T_{Q}(x^{n}-y_{k}^{n})=T_{Q}(\underset{i=1}{\overset{k}{\vee}%
}(x^{n}-x_{i}^{n}))\\
&  \leq T_{Q}(\overset{k}{\underset{i=1}{%
{\displaystyle\sum}
}}(x^{n}-x_{i}^{n}))=\overset{k}{\underset{i=1}{%
{\displaystyle\sum}
}}T_{Q}(x^{n}-x_{i}^{n})\\
&  \leq\overset{k}{\underset{i=1}{%
{\displaystyle\sum}
}}\frac{\epsilon}{2^{i}}<\epsilon.
\end{align*}
Which implies that
\[
Q(x)=0.
\]
Thus $C_{P}\subset N_{Q}$.

Now for $2\Rightarrow3$ from $C_{P}\subset N_{Q}$ we have
\[
N_{Q}^{d}=C_{Q}\subset C_{P}^{d}=N_{P}^{dd}=N_{P}\text{ (because }N_{P}\text{
is a band).}%
\]

For the implication $3\Rightarrow4$. From $C_{Q}\subset N_{P}$ we have
$C_{P}\perp C_{Q}$ because $N_{P}\perp C_{P}$.

Now for the implication $4\Rightarrow1$. It follows from $C_{P}\perp C_{Q}$
that $C_{Q}\subset C_{P}^{d}=N_{P}$. Now if $x=y+z\in N_{Q}\oplus C_{Q}$ we
have
\[
0\leq(P\wedge Q)(x)=(P\wedge Q)(y)+(P\wedge Q)(z)\leq Q(y)+P(z)=0.
\]
Thus $P\wedge Q=0$ holds on the order dense ideal $N_{P}\oplus C_{Q}$ on the
other hand $P\wedge Q$ is order continuous then it follows that
\[
P\wedge Q=0\text{.}%
\]
and the proof is finished.
\end{proof}

The following Proposition will be of great use next.

\begin{proposition}
Let $E$ be an Archimedean Riesz space and $(B_{i})_{i\in I}$ a set of bands in
$E$ then we have
\[
(\underset{i\in I}{\cap}B_{i})^{d}=(\underset{i\in I}{%
{\displaystyle\sum}
}B_{i}^{d})^{dd}%
\]

\end{proposition}

\begin{proof}
It is clear that it is sufficient to prove that if $card(I)=2$. Now let us
prove that $(B_{1}+B_{2})^{d}=B_{1}^{d}\cap B_{2}^{d}$. It follows from
$B_{1}\subset B_{1}+B_{2}$ and $B_{2}\subset B_{1}+B_{2}$ that $(B_{1}%
+B_{2})^{d}\subset B_{1}^{d}\cap B_{2}^{d}$. For the converse inclusion, let
$x\in B_{1}^{d}\cap B_{2}^{d}$ and $y+z\in B_{1}+B_{2}$. In view of the vector
lattice inequality $0\leq\left\vert x\right\vert \wedge\left\vert
y+z\right\vert \leq\left\vert x\right\vert \wedge(\left\vert y\right\vert
+\left\vert z\right\vert )\leq\left\vert x\right\vert \wedge\left\vert
y\right\vert +\left\vert x\right\vert \wedge\left\vert z\right\vert =0$ we
have $x\perp(y+z)$ which implies that $B_{1}^{d}\cap B_{2}^{d}\subset
(B_{1}+B_{2})^{d}.$ Therefore
\[
(B_{1}+B_{2})^{d}=B_{1}^{d}\cap B_{2}^{d}\text{.}%
\]
By replacing $B_{1}$ by $B_{1}^{d}$, $B_{2}$ by $B_{2}^{d}$ (because are also
a bands) we have
\[
(B_{1}^{d}+B_{2}^{d})^{dd}=(B_{1}^{dd}\cap B_{2}^{dd})^{d}=(B_{1}\cap
B_{2})^{d}\text{.}%
\]
The general case is left to the reader. The proof is finished.
\end{proof}

As an immediate application to the above proposition we obtain the following result.

\begin{theorem}
Let $P\in\mathcal{P}_{ob}(^{n}E,C(Y))$ where $C(Y)$ is the space of continuous
functions in a topological Hausdorff space $Y$. Then
\[
N_{P}=\underset{y\in Y}{\cap}N_{\delta_{y}\circ P}\text{ and }C_{P}%
=(\underset{y\in Y}{%
{\displaystyle\sum}
}C_{\delta_{y}\circ P})^{dd}%
\]
where $\delta_{y}$ is the evaluation at $y$.
\end{theorem}

\begin{proof}
Let $x\in N_{P}$ thus $\left\vert P\right\vert (\left\vert x\right\vert )=0$
which implies that $\left\vert P\right\vert (\left\vert x\right\vert )(y)=0$
for all $y\in Y$ then $\delta_{y}\circ\left\vert P\right\vert (\left\vert
x\right\vert )=0$. Now by using the fact that $\delta_{y}$ is a lattice and
order continuous operator we have%
\[
\delta_{y}\circ\left\vert P\right\vert =\left\vert \delta_{y}\circ
P\right\vert .
\]
Thus $\left\vert \delta_{y}\circ P\right\vert \left\vert x\right\vert =0.$
Which implies that $x\in N_{\delta_{y}\circ P}$ for all $y\in Y$. The converse
inclusion is obviously. Therefore
\[
N_{P}=\underset{y\in Y}{\cap}N_{\delta_{y}\circ P}\text{.}%
\]
Consequently by using the preceding Proposition we have
\[
C_{P}=N_{P}^{d}=(\underset{y\in Y}{\cap}N_{\delta_{y}\circ P})^{d}=(%
{\displaystyle\sum}
(N_{\delta_{y}\circ P})^{d})^{dd}=(%
{\displaystyle\sum}
C_{\delta_{y}\circ P})^{dd}\text{.}%
\]

\end{proof}

Next we give a Nakano type theorem in the space $\mathcal{P}_{oc}%
(^{n}E,C(Y))^{+}$.

\begin{theorem}
Let $P,Q\in\mathcal{P}_{ob}(^{n}E,C(Y))^{+}$ such that $P$ or $Q$ is order
continuous. Then following

\begin{enumerate}
\item $C_{P}\perp C_{Q}$

\item $P\perp Q$

\item $\delta_{y}\circ P\perp\delta_{y}\circ Q$ for all $y\in Y$.

\item $C_{\delta_{y}\circ P}\perp C_{\delta_{y}\circ Q}$ for all $y\in Y$.
\end{enumerate}

satisfies $1\Rightarrow2\Rightarrow3\Leftrightarrow4$.
\end{theorem}

\begin{proof}
The implication $1\Rightarrow2$ follows with the same techniques that prove
the implication \ $1\Rightarrow2$ in Theorem 3.

For $2\Rightarrow3\Leftrightarrow4$. Let $y\in Y$ and $\delta_{y}$ the
evaluation at $y$, from the fact that $P\perp Q$ and the order continuity of
$\delta_{y}$ we have $\delta_{y}\circ P\perp\delta_{y}\circ Q$. Thus from
Theorem 3 we have $3\Rightarrow4.$
\end{proof}

It is a natural question to ask whether the implication $2\Rightarrow1$ and
$3\Rightarrow1$ are true. The answer is negative and this illustrated by the
following example.

\begin{example}
Consider the two positive order continuous orthogonally additive
$2$-homogeneous polynomials $P,Q:L_{2}[0,1]\rightarrow L_{2}[0,1]$ defined by
\[
P(f)=(\int_{0}^{1}f(x)^{2}dx)\mathcal{X}_{[0,\dfrac{1}{2}]}\text{ and
}Q(f)=(\int_{0}^{1}f(x)^{2}dx)\mathcal{X}_{[\dfrac{1}{2},1]}.
\]
Observe that $P\perp Q$ because for every $0\leq f\in L_{2}[0,1]$ we have
\[
0\leq(P\wedge Q)(f)\leq P(f)\wedge Q(f)=0.
\]
Thus $P\wedge Q=0$ holds in $\mathcal{P}_{oc}(^{2}L_{2}[0,1])$. However
\[
N_{S}=N_{T}=\left\{  0\right\}  \text{ and so }C_{T}=C_{S}=L_{2}[0,1]
\]
proving that $C_{T}$ and $C_{S}$ are not disjoints sets proving that $C_{T}$
and $C_{S}$ are not disjoints sets. On the other hand it is clear that for all
$y\in\lbrack0,1]$ we have
\[
\delta_{y}\circ P=0\text{ or }\delta_{y}\circ Q=0
\]
which implies that $\delta_{y}\circ P\perp\delta_{y}\circ Q$ but $C_{T}%
=C_{S}=L_{2}[0,1]$.
\end{example}

So by the preceding example we can confirm that Nakano theorem fails in the
case of orthogonally additive polynomials between general Riesz spaces. It is
worth recording that by Theorem 5 we can do a little better for the
implication $4\Rightarrow1$ as illustrated in the following result.

\begin{theorem}
Let $P,Q\in\mathcal{P}_{oc}(^{n}E,C(Y))^{+}$ where here $E$ is an Archimedean
Riesz spaces Then the following are equivalents:

\begin{enumerate}
\item $C_{P}\perp C_{Q}$

\item $C_{\delta_{y}\circ P}\perp C_{\delta_{z}\circ Q}$ for all $y,z\in Y$.
\end{enumerate}
\end{theorem}

\begin{proof}
For $1\Rightarrow2$ from Theorem 5 it is clear that we have $C_{\delta
_{y}\circ P}\subset C_{P}$ and $C_{\delta_{z}\circ Q}\subset C_{Q}$. Now by
using the fact that $C_{P}\perp C_{Q}$ we have $C_{\delta_{y}\circ P}\perp
C_{\delta_{z}\circ Q}$.

For $2\Rightarrow1$. If $C_{\delta_{y}\circ P}\perp C_{\delta_{z}\circ Q}$ for
all $y,z\in Y$ then $%
{\displaystyle\sum}
C_{\delta_{y}\circ P}\perp%
{\displaystyle\sum}
C_{\delta_{y}\circ Q}$. Consequently $%
{\displaystyle\sum}
C_{\delta_{y}\circ P}\perp%
{\displaystyle\sum}
C_{\delta_{y}\circ Q}$ which implies from Theorem 5 that
\[
C_{P}=(%
{\displaystyle\sum}
C_{\delta_{y}\circ P})^{dd}\perp(%
{\displaystyle\sum}
C_{\delta_{y}\circ Q})^{dd}=C_{Q}%
\]
and the proof is finished.
\end{proof}

Next, we give from the theory of homogeneous orthogonally additive polynomials
a new concept, so called, order polynomial dual for Archimedean Riesz space.
It will be noted that the motivation for generalization of the classical order
dual came by and large from non linear functional analysis and its
applications. Indeed, it came to light that the classical order dual fails to
describe certain very interesting property of the Riesz space. In this
prospect, we introduce and give a systematic study of the order polynomial
dual for Archimedean Riesz space. This study is perhaps not quite in final
form yet. But we believe that our results in this direction are a little more
general that previously known results. surprisingly enough, to the best of our
knowledge no attention at all has been paid in the literature to this concept.

The vector space $\mathcal{P}_{oc}(^{n}E,%
\mathbb{R}
)$ of all order continuous $n$-homogeneous orthogonally additive polynomials
will be called the \textit{order continuous }$n$-homogeneous\textit{
orthogonally additive polynomial} \textit{dual} of $E$ and will be denoted by
$(^{n}E)^{\sim_{p}}$. Since $%
\mathbb{R}
$ is a Dedekind complete Riesz space, it follows at once from [] that
$(^{n}E)^{\sim_{p}}$ is precisely the vector space generated by the positive
orthogonally additive $n$-homogeneous polynomials. Moreover $(^{n}E)^{\sim
_{p}}$ is a Dedekind complete Riesz space. In general, there is no guarantee
that a Riesz space supports any non trivial orthogonally additive polynomials
of degree greater than one. Thus the order polynomial dual of $E$ may happen
to be the trivial space. As an example we have already mentioned in the
introduction that the Riesz space $(^{n}L_{p}[0,1])^{\sim_{p}}$ is
isometrically isomorphie to $L_{q}[0,1]$ for all $n<p$. When $n>p$ there are
no non zero $n$-homogeneous orthogonally additive polynomials on $L_{p}[0,1]$.
Therefore $(^{n}L_{p}[0,1])^{\sim_{p}}=\left\{  0\right\}  $ for all $n>p$.

Now from the fact that $(^{n}E)^{\sim_{p}}$ is again a Riesz space. Thus we
can consider the order dual of $(^{n}E)^{\sim_{p}}$ which is $((^{n}%
E)^{\sim_{p}})^{\sim}=\mathcal{L}_{b}((^{n}E)^{\sim_{p}},%
\mathbb{R}
)$, the space of order bounded linear functional on $(^{n}E)^{\sim_{p}}$. For
each $x\in E$ an order bounded linear functional $\widehat{x}$ can be defined
on $(^{n}E)^{\sim_{p}}$ via the formula
\[
\widehat{x}(P)=P(x)\text{ for all }P\in(^{n}E)^{\sim_{p}}\text{.}%
\]
Clearly, $x\geq0$ implies $\widehat{x}\geq0$. Also, since $P_{\alpha
}\downarrow0$ in $(^{n}E)^{\sim_{p}}$ holds if and only if $P_{\alpha
}(x)\downarrow0$ in $%
\mathbb{R}
$ for all $x\in E^{+}.$ Then $\widehat{x}$ is order continuous linear
functional on $(^{n}E)^{\sim_{p}}$. Thus a positive map $x\rightarrow
\widehat{x}$ can be defined from $E$ to $(((^{n}E)^{\sim_{p}})^{\sim})_{n}$,
the space of order continuous linear functional on $(^{n}E)^{\sim_{p}}$. This
map is called the canonical embedding of $E$ into $(((^{n}E)^{\sim_{p}}%
)^{\sim})_{n}$ which is one to one when $(^{n}E)^{\sim_{p}}$ separates the
points of $E$, that is for all $x\neq y$ there exists $P\in(^{n}E)^{\sim_{p}}$
such that $P(x)\neq P(y)$. It should be noted that this canonical embedding
does not necessarily preserve finite suprema and infima as it is proved in
next Proposition;

\begin{proposition}
Let $E$ an Archimedean Riesz space and $x\in E$. Then%
\[
\widehat{x}^{+}=\left\{
\begin{array}
[c]{ccc}%
\widehat{x} & \text{if} & n\text{ is even}\\
\widehat{x^{+}} & \text{if} & n\text{ is odd}%
\end{array}
\right.
\]

\end{proposition}

\begin{proof}
Consider $P\in((^{n}E)^{\sim_{p}})^{+}$ and $x\in E$. According to [2, Theorem
1.23] we can see that
\begin{align*}
\widehat{x}^{+}(P)  &  =\sup\left\{  Q(x):Q\in\mathcal{P}_{ob}(^{n}E,%
\mathbb{R}
)^{+}\text{ such that }0\leq Q\leq P\right\} \\
&  =\sup\left\{  T_{Q}(x^{n}):T_{Q}\in\mathcal{L}_{b}\mathcal{(}\underset
{i=!}{\overset{n}{\Pi}}E^{ru},%
\mathbb{R}
)^{+}\text{ such that }0\leq T_{Q}\leq T_{P}\right\} \\
&  =T_{P}((x^{n})^{+})
\end{align*}
First if $n$ is even then $x^{n}\geq0$ therefore $T_{P}((x^{n})^{+}%
)=T_{P}(x^{n})=P(x)=\widehat{x}(P)$. Secondly if $n$ is odd writing
\begin{align*}
T_{P}((x^{n})^{+})  &  =T_{P}(((x^{+}-x^{-})^{n})^{+})=T_{P}(((x^{+}%
)^{n}-(x^{-})^{n}))^{+})\\
&  =T_{P}((x^{+})^{n})=P(x^{+})=\widehat{x^{+}}(P).
\end{align*}
Thus
\[
\widehat{x}^{+}=\left\{
\begin{array}
[c]{ccc}%
\widehat{x} & \text{if} & n\text{ is even}\\
\widehat{x^{+}} & \text{if} & n\text{ is odd}%
\end{array}
\right.
\]
and the proof is finished.
\end{proof}

Before showing our main second result of the paper we need the two following lemmas.

\begin{lemma}
The canonical embedding of $E$ into $(((^{n}E)^{\sim_{p}})^{\sim})_{n}$ is an
order bounded orthogonally additive $n$-homogeneous polynomial. Which is
increasing map at $E^{+}$. Therefore:

\begin{enumerate}
\item For every $x\in E$ and $\lambda\in%
\mathbb{R}
$ we have $\widehat{\lambda x}=\lambda^{n}\widehat{x}$.

\item For every $x,y\in E$ such that $x\perp y$ we have $\widehat
{x+y}=\widehat{x}+\widehat{y}.$

\item If $0\leq x\leq y$ then $\widehat{0}=0\leq\widehat{x}\leq\widehat{y}.$
\end{enumerate}
\end{lemma}

\begin{proof}
For (1), let $x\in E$, $P\in(^{n}E)^{\sim_{p}}$ and $\lambda\in%
\mathbb{R}
$ we have $\widehat{\lambda x}(P)=P(\lambda x)=\lambda^{n}P(x)=\lambda
^{n}\widehat{x}(P)$ which implies that $\widehat{\lambda x}=\lambda
^{n}\widehat{x}$. Now for (2), let $x,y\in E$ such that $x\perp y$ and
$P\in(^{n}E)^{\sim_{p}}$ we have $\widehat{x+y}(P)=P(x+y)=P(x)+P(y)=(\widehat
{x}+\widehat{y})(P)$ which implies that $\widehat{x+y}=\widehat{x}+\widehat
{y}.$ Now for (3), let $P\in((^{n}E)^{\sim_{p}})^{+}$, $T_{P}$ its associated
linear operator and $0\leq x\leq y$ from the fact that $E^{u}$ is an
$f$-algebra it follows that $0\leq x^{n}\leq y^{n}$. Thus $T_{P}(x^{n})\leq
T_{P}(y^{n})$ (because $T_{P}$ is positive) which implies that $P(x)=\widehat
{x}(P)\leq P(y)=\widehat{y}(P)$. Therefore
\[
0\leq\widehat{x}\leq\widehat{y}\text{,}%
\]
and the proof is finished.
\end{proof}

The above lemma yields to the following lemma.

\begin{lemma}
Let $x,y\in E^{+\text{ }}$then%
\[
0\leq\widehat{x\wedge y}\leq\widehat{x}\wedge\widehat{y}\text{.}%
\]

\end{lemma}

Now we are able to announce the second main result of the paper which is
somehow a Nakano type theorem for orthogonally additive homogeneous
polynomials. We hope will be of great contribution to the literatures.

\begin{theorem}
Let $E$ be an Archimedean Riesz space. Then the embedding map%
\[%
\begin{array}
[c]{cccc}%
\wedge: & E & \rightarrow & (((^{n}E)^{\sim_{p}})^{\sim})_{n}\\
& x & \rightarrow & \widehat{x}%
\end{array}
\]
is an order continuous orthogonally additive $n$-homogeneous polynomial whose
range is order dense in $(((^{n}E)^{\sim_{p}})^{\sim})_{n}$, the order
continuous dual of $\mathcal{P}_{oc}(^{n}E,%
\mathbb{R}
)$.
\end{theorem}

\begin{proof}
We have already mentioned that the canonical embedding of $E$ into
$((^{n}E)^{\sim_{p}})^{\sim}$ is an order continuous orthogonally additive
$n$-homogeneous polynomial. So we need only to prove that the set $\left\{
\widehat{x},x\in E\right\}  $ is order dense in $(((^{n}E)^{\sim_{p}})^{\sim
})_{n}.$ To this end, let $0<\phi\in(((^{n}E)^{\sim_{p}})^{\sim})_{n}$ so by
the order continuity of $\phi$ and the fact that $\phi\neq0$ there exists
$0<P\in C_{\phi}=N_{\phi}^{d}$. Now, from the fact that $P$ is order
continuous and $P\neq0$ it follows that $C_{P}\neq\left\{  0\right\}  .$ So
pick $0<x\in C_{P}.$ if $\widehat{x}\wedge\phi=0$ then by the classical Nakano
theorem we have $\widehat{x}(C_{\phi})=0$. Consequently, $P(x)=0$ which
implies that $x\in N_{P}\cap C_{P}=\left\{  0\right\}  $ which is impossible.
Then $\widehat{x}\wedge\phi>0$. Let $0<\phi_{1}=\widehat{x}\wedge\phi
\leq\widehat{x}$ holds in $((^{n}E)^{\sim_{p}})^{\sim}$ for some $x\in E$. Now
from the fact that $((^{n}E)^{\sim_{p}})^{\sim}$ is an Archimedean Riesz space
and the fact that $\widehat{x}\neq0$ there exists let $0<\epsilon<1$ such
that
\[
\psi=\phi_{1}\vee\epsilon\widehat{x}-\epsilon\widehat{x}=(\phi_{1}%
-\varepsilon\widehat{x})^{+}>0\text{.}%
\]
By the order continuity of $\psi$ there exists $0<P\in C_{\psi}$. Now it
follows from the order continuity of $P$ that $C_{P}\neq\left\{  0\right\}  $.
We claim that we can choice $0<y\in C_{P}$ such that $z=y\wedge\epsilon x>0$.
Assume by way of contradictions that $\epsilon x\in C_{P}^{d}=N(P)$ which
implies that $\widehat{x}(P)=P(x)=0.$ On the other hand $0\leq\psi\leq\phi
_{1}=\widehat{x}\wedge\phi\leq\widehat{x}$ hence $0\leq\psi(P)\leq\widehat
{x}(P)=0$. Thus $\psi(P)=0$ therefore $P\in C_{\psi}\cap N_{\psi}=\left\{
0\right\}  $ then $P=0$ which is impossible. So there exists $0<y\in C_{P}$
such that $z=y\wedge\epsilon x>0$. Now, We claim that $z$ satisfies
\[
0<\widehat{z}\leq\phi_{1}%
\]
in $((^{n}E)^{\sim_{p}})^{\sim}$. First to see that $0<\widehat{z}$ holds,
note that $0\leq z\leq y$ and $y\in C_{P}$ which is an ideal. So $0<z\in
C_{P}$ which implies that $P(z)\neq0$ so $\widehat{z}(P)\neq0$, thus
$0<\widehat{z}$. Now, we claim that $\widehat{z}\leq\phi_{1}$. To this end,
assume by way of contradictions that $\theta=(\widehat{z}-\phi_{1})^{+}>0$. By
the order continuity of $\theta$ there exists $0<Q\in C_{\theta}$. On the
other hand by Lemma we have $\widehat{z}\leq\varepsilon\widehat{x}$ because
$z\leq\varepsilon x$ which implies that
\[
0\leq\theta=(\widehat{z}-\phi_{1})^{+}\leq(\varepsilon\widehat{x}-\phi
_{1})^{+}=(\phi_{1}-\varepsilon\widehat{x})^{-}\text{.}%
\]
Which implies that $\theta\bot\psi$ so $C_{\theta}\bot C_{\psi}$. On the other
hand $P\in C_{\psi}$, $Q\in C_{\theta}$ which implies by theorem that $P\bot
Q$ again by the same theorem [] we get
\[
Q(C_{P})=\left\{  0\right\}  .
\]
Therefore by applying lemma once more and from the fact that $z\leq y$ we
have
\[
0<\theta(Q)=(\widehat{z}-\phi_{1})^{+}(Q)\leq\widehat{z}(Q)\leq\widehat
{y}(Q)=Q(y)=0\text{.}%
\]
Thus $\theta(Q)=0$ which implies that $Q\in N_{\theta}\cap C_{\theta}=\left\{
0\right\}  $ then $Q=0$ which is impossible. Hence
\[
0<\widehat{z}\leq\phi_{1}\leq\phi\text{.}%
\]
Consequently the set $\left\{  \widehat{x},x\in E\right\}  $ is order dense in
$(((^{n}E)^{\sim_{p}})^{\sim})_{n}$, and the proof of the theorem is complete.
\end{proof}


\begin{thebibliography}{99}                                                                                               %


\bibitem {}Alaminos, J., Extremera, J., Villena, A. R., Orthogonally additive
polynomials on Fourier algebras. J. Math. Anal. Appl. 422 (2015) no. 1, 72-83.

\bibitem {}Aliprantis, C.D., Burkinshaw, O., Positive operators, Academic
Press, Orlando, 1985.

\bibitem {}Ben Amor, F., Orthogonally additive homogenous polynomials on
vector lattices. Comm. Algebra 43 (2015), no. 3, 1118-1134.

\bibitem {}Benyamini, Y., Lassalle, S., Llavona, J.L.G., Homogeneous
orthogonally-additive polynomials on Banach lattices, Bulletin of the London
Mathematical Society 38 (3) (2006), 459-469.

\bibitem {}Boulabiar, K., On products in lattice-ordered algebras, Journal of
the Australian Mathematical Society 75(1) (2003), 1435-1442.

\bibitem {}Boyd, C., Raymond A. R., and Snigireva, N., A Nakano carrier
theorem for polynomials. Proc. Amer. Math. Soc. 151 (2023), 1621-1635.

\bibitem {}Bu, Q., Buskes, G., Polynomials on Banach lattices and positive
tensor products, J. Math. Anal. Appl. 388 (2012) 845-862.

\bibitem {}Buskes, G., Kusraev A.G., Representation and extension of
orthoregular bilinear operators, Vladikavkazski i Matematicheski i Zhurnal 9
(2007), 16-29.

\bibitem {}Carand, D., Lassalle, S., Zalduendo, I., Orthogonally additive
polynomials over $C(K)$ are measures a short proof, Integral Equations and
Operator Theory, 56(4)(2006), 597-602.

\bibitem {}Chil, E., and Dorai A., Continuous orthosymmetric multilinear maps
and homogeneous polynomials on Riesz spaces. J., of Analysis (2020) 28 1127-1141.

\bibitem {}Grecu, B.C., Ryan, R.A., Polynomials on Banach spaces with
unconditional bases, Proc. Amer. Math. Soc. 133 (2005) 1083-1091.

\bibitem {}Kusraeva, Z. A., Representation of orthogonally additive
polynomials. Siberian Mathematical Journal 52 (2011) No. 2 248-255.

\bibitem {}Ibort, A., Linares, P., Llavona, J.G., A representation theorem for
orthogonally additive polynomials on Riesz spaces. Rev. Mat. Complut. 25
(2012), 21-30.

\bibitem {}Loane, J., Polynomials on vector lattices, PhD dissertation,
National University of Ireland, Galway, 2007.

\bibitem {}Luxemburg, W. A., Zaanen, A.C., Riesz spaces I, North-Holland.
Amsterdam, 1971.

\bibitem {}De pagter, B., $\ f$-algebras and Orthomorphisms, thesis, Leiden, 1981.

\bibitem {}Palazuelos C., Peralta A.M., Villanueva I., Orthogonally additive
polynomials on $%
\mathbb{C}
^{\ast}$ algebras, Quart. J. Math. Anal. Appl. 59 (2008), 363-374.

\bibitem {}P\'{e}rez-Garcia, D., Villanueva, I., Orthogonally additive
polynomials on space of continuous functions, Journal of Mathematical Analysis
and Applications 306 (1) (2005), 97-105.

\bibitem {}Schaefer, H.H., Banach lattices and Positive Operators, Grundlehren
Math. Wiss., vol. 215, Springer, Berlin, 1974.

\bibitem {}Sundaresan, K., Geometry of spaces of homogeneous polynomials on
Banach lattices, Applied geometry and discrete mathematics, 571-586, DIMACS
Series in Discrete Mathematics and Theoretical computer Science no. 4,
American mathematical Society, Providence, RI, 1991.

\bibitem {}Toumi, M.A., Orthogonality of polynomials and orthosymmetry. Forum
Math. 27 (2015), no. 3, 1389-1400.
\end{thebibliography}
\end{document}